\DeclareSymbolFont{AMSb}{U}{msb}{m}{n}
\documentclass[11pt,a4paper,leqno,noamsfonts]{amsart}
\makeatletter 
\newcommand{\mylabel}[2]{#2\def\@currentlabel{#2}\label{#1}}
\renewcommand\@biblabel[1]{#1.}
\makeatother
\usepackage[english]{babel}
\usepackage[dvipsnames]{xcolor}
\usepackage{graphicx,pifont} 
\usepackage[utopia]{mathdesign}
\usepackage[a4paper,top=3cm,bottom=3cm,
           left=3.5cm,right=3.5cm,marginparwidth=60pt]{geometry}
\usepackage[utf8]{inputenc}
\usepackage{braket,caption,comment,mathtools,stmaryrd}
\usepackage[usestackEOL]{stackengine}
\usepackage{multirow,booktabs,microtype,relsize}
\usepackage[colorlinks,bookmarks]{hyperref} %
      \hypersetup{colorlinks,%
            citecolor=britishracinggreen,%
            filecolor=black,%
            linkcolor=cobalt,%
            urlcolor=cornellred}
      \numberwithin{equation}{section}
\usepackage[capitalise]{cleveref}
\DeclareMathOperator{\opp}{op}
\DeclareMathOperator{\Sch}{Sch}

\DeclareSymbolFont{usualmathcal}{OMS}{cmsy}{m}{n}
\DeclareSymbolFontAlphabet{\mathcal}{usualmathcal}
\DeclareMathAlphabet\BCal{OMS}{cmsy}{b}{n}

\definecolor{cornellred}{rgb}{0.7, 0.11, 0.11}
\definecolor{britishracinggreen}{rgb}{0.0, 0.26, 0.15}
\definecolor{cobalt}{rgb}{0.0, 0.28, 0.67}

\usepackage[colorinlistoftodos]{todonotes} 
\usepackage{amsmath,float}
\usetikzlibrary{decorations.markings}

\newcommand{\bfk}{\mathbf{k}}

\newcommand{\ncHilb}{\mathrm{ncHilb}}
\newcommand{\ncQuot}{\mathrm{ncQuot}}

\newcommand{\BA}{{\mathbb{A}}}

\newcommand{\BC}{{\mathbb{C}}}

\newcommand{\BL}{{\mathbb{L}}}

\newcommand{\BP}{{\mathbb{P}}}
\newcommand{\BQ}{{\mathbb{Q}}}

\DeclareMathOperator{\Sets}{Sets}
\DeclareMathOperator{\Quot}{Quot}

\DeclareMathOperator{\Sym}{Sym}

\DeclareMathOperator{\Coh}{Coh}

\DeclareMathOperator{\Span}{Span}

\DeclareMathOperator{\vir}{\mathrm{vir}}

\DeclareMathOperator{\Var}{Var}

\DeclareMathOperator{\GL}{GL}

\DeclareMathOperator{\Mat}{Mat}

\DeclareMathOperator{\St}{St}

\newcommand{\into}{\hookrightarrow}
\newcommand{\onto}{\twoheadrightarrow}

\DeclareFontFamily{OT1}{rsfs}{}
\DeclareFontShape{OT1}{rsfs}{n}{it}{<-> rsfs10}{}
\DeclareMathAlphabet{\curly}{OT1}{rsfs}{n}{it}

\newcommand\End{\operatorname{End}}

\DeclareMathOperator{\lHom}{\mathscr{H}\kern-0.3em\mathit{om}}
\DeclareMathOperator{\RRlHom}{\mathbf{R}\kern-0.025em\mathscr{H}\kern-0.3em\mathit{om}}

\DeclareMathOperator{\lExt}{{\mathscr{E}\kern-0.2em\mathit{xt}}}
\DeclareMathOperator{\Rcrit}{\mathbf{R}\kern-0.0em\mathrm{crit}}

\newcommand\Spec{\operatorname{Spec}}

\newcommand\id{\operatorname{id}}

\newcommand{\HH}{\mathrm{H}}
\newcommand{\OO}{\mathscr O}



\usepackage[all]{xy}
\usepackage{tikz}
\usepackage{tikz-cd}
\usepackage{adjustbox}
\usepackage{rotating}
\usepackage{comment}

\usetikzlibrary{matrix,shapes,intersections,arrows,decorations.pathmorphing}
\tikzset{commutative diagrams/arrow style=math font}
\tikzset{commutative diagrams/.cd,
mysymbol/.style={start anchor=center,end anchor=center,draw=none}}

\tikzset{
shift up/.style={
to path={([yshift=#1]\tikztostart.east) -- ([yshift=#1]\tikztotarget.west) \tikztonodes}
}
}

\theoremstyle{definition}

\newtheorem*{lemma*}{Lemma}
\newtheorem*{theorem*}{Theorem}
\newtheorem*{example*}{Example}
\newtheorem*{fact*}{Fact}
\newtheorem*{notation*}{Notation}
\newtheorem*{definition*}{Definition}
\newtheorem*{prop*}{Proposition}
\newtheorem*{remark*}{Remark}
\newtheorem*{corollary*}{Corollary}

\newtheorem*{conventions*}{Conventions}

\newtheorem{definition}{Definition}[section]

\newtheorem{example}[definition]{Example}

\newtheorem{remark}[definition]{Remark}

\newtheoremstyle{thm} 
        {3mm}
        {3mm}
        {\slshape}
        {0mm}
        {\bfseries}
        {.}
        {1mm}
        {}
\theoremstyle{thm}
\newtheorem{theorem}[definition]{Theorem}

\newtheorem{lemma}[definition]{Lemma}
\newtheorem{prop}[definition]{Proposition}

\newtheoremstyle{ex} 
        {3mm}
        {3mm}
        {}
        {0mm}
        {\scshape}
        {.}
        {1mm}
        {}
\theoremstyle{ex}

\newtheoremstyle{sol} 
        {3mm}
        {3mm}
        {}
        {0mm}
        {\scshape}
        {.}
        {1mm}
        {}
\theoremstyle{sol}

\newtheorem*{Acknowledgments*}{Acknowledgments}

\title[Motivic classes of noncommutative Quot schemes]{Motivic classes of noncommutative Quot schemes}
\author{Andrea T. Ricolfi}

\begin{document}

\maketitle

\begin{abstract}
We provide a recursive formula for the motivic class of the noncommutative Quot scheme in the Grothendieck ring of stacks.
\end{abstract}


\section{Introduction}

\subsection{Overview}
We work over a fixed algebraically closed field $\bfk$. Fix integers $r,d>0$ and $n\geq 0$. Given an $n$-dimensional $\bfk$-vector space $V_n$, form the $(dn^2+rn)$-dimensional affine space
\[
\BA=\End_{\bfk}(V_n)^d \times V_n^r.
\]
The general linear group $\GL_n$ (a smooth algebraic group over $\bfk$) acts on $\BA$ by simultaneous conjugation on the matrices and in the natural way on the vectors. Consider the open subscheme $U^n_{r,d} \into \BA$ parametrising tuples $(T_1,\ldots,T_d,v_1,\ldots,v_r) \in \BA$ such that the $\bfk$-linear span of all monomials in $T_1,\ldots,T_d$ applied to $v_1,\ldots,v_r$ coincides with the whole $V_n$. Then $\GL_n$ acts \emph{freely} on $U^n_{r,d}$. The quotient
\[
\ncQuot^n_{r,d} = U^n_{r,d}/\GL_n
\]
is the \emph{noncommutative Quot scheme}. It is a smooth quasiprojective $\bfk$-variety of dimension $N_{n,r,d} = (d-1)n^2+rn$. See \Cref{sec:moduli-space} for further details on this space. 

Let $K_0(\Var_{\bfk})$ be the Grothendieck ring of $\bfk$-varieties, and $K_0(\St_{\bfk})$ the Grothendieck ring of $\bfk$-stacks \cite{EKStacks}. In this paper we compute the motive of $\ncQuot^n_{r,d}$, more precisely the image of the class $[\ncQuot^n_{r,d}] \in K_0(\Var_{\bfk})$ along the natural map $K_0(\Var_{\bfk}) \to K_0(\St_{\bfk})$.

\begin{theorem}\label{thm:main-thm}
Fix integers $r,d>0$ and $n\geq 0$. Then, there is an identity
\begin{equation}\label{eqn:main-formula}
[\ncQuot^n_{r,d}] = \frac{\BL^{dn^2}}{[\GL_n]}\left(\BL^{rn} - \sum_{k=0}^{n-1}\, [\ncQuot^k_{r,d}][\GL_k][G(k,n)]\BL^{-dnk} \right)
\end{equation}
in $K_0(\St_{\bfk})$, where $\BL = [\BA^1_{\bfk}]$ is the Lefschetz motive. Moreover, if $\bfk=\BC$, the generating functions
\[
\mathsf Z_{r,d}(t) = \sum_{n=0}^\infty\, \BL^{-N_{n,r,d}/2} [\ncQuot^{n}_{r,d}] (\BL^{(r-1)/2}t)^n \in K_0(\St_{\BC}) \llbracket t \rrbracket
\]
satisfy the functional equations
\begin{align*}
    \mathsf Z_{1,d}(t) &= 1+\BL^{d/2}t \prod_{i=0}^{d-1}\mathsf Z_{1,d}(\BL^it)\\
\mathsf Z_{r,d}(t) &= \prod_{i=0}^{r-1} \mathsf Z_{1,d}(\BL^{i}t).
\end{align*}
\end{theorem}

Note that \Cref{eqn:main-formula} is recursive but explicit, being determined by the $n=0$ motive, which is just the ring identity $1 = [\Spec \bfk]$ for any fixed $r$ and $d$; moreover, the motive of $\GL_n$ is the class 
\[
\left[\GL_n\right] = \prod_{i=0}^{n-1} \,\left(\BL^n-\BL^i\right) = \BL^{\binom{n}{2}}[n]_{\BL} \in K_0(\Var_{\bfk}),
\]
where we have set $[n]_x = \prod_{1\leq i\leq n}(x^i-1)$ for a formal variable $x$. Similarly, the motive of the Grassmannian $G(k,n)$ is 
\[
[G(k,n)] = \frac{[n]_{\BL}}{[k]_{\BL}[n-k]_{\BL}} \in K_0(\Var_{\bfk}).
\]
\Cref{eqn:main-formula} is proved in \Cref{sec:computation}. The functional equations in \Cref{thm:main-thm} are directly derived (cf.~\Cref{prop:functional-eqn}) from Reineke's functional equations for the Poincar\'e polynomial of $\ncQuot^n_{r,d}$, see \cite[Thm.~1.4]{Reineke_2005}.

\subsection{Related works}
If $r=1$, the space $\ncQuot^n_{1,d}$ is called the \emph{noncommutative Hilbert scheme}, introduced by Nori \cite{Nori1} and studied e.g.~by Larsen--Lunts \cite{NCHilb1}. If $d=3$, then the space $\ncQuot^n_{r,3}$ has a key role in Donaldson--Thomas theory and its (motivic, K-theoretic) refinements \cite{FMR_K-DT,cazzaniga2020higher}.
Grothendieck's Quot scheme $\Quot_{\BA^d}(\OO^{\oplus r},n)$, parametrising quotients $\OO^{\oplus r} \onto E$ such that $E \in \Coh(\BA^d)$ is a sheaf of finite support with $h^0(E)=n$, sits inside $\ncQuot^n_{r,d}$ as the locus cut out by the vanishing of the commutators $[T_i,T_j] = 0$.  For links to the context of moduli spaces of sheaves (or framed sheaves), see for instance \cite{Jardim} for the $r=1$ case and \cite{Henni_Quot,cazzaniga2020framed} for arbitrary $r$.

As already mentioned, a functional equation for the generating series of Poincar\'e polynomials of $\ncQuot^n_{r,d}$ was found by Reineke \cite{Reineke_2005}. See also \cite[Cor.~5.3]{Engel_2008} for a recursion, at least formally similar to our \Cref{eqn:main-formula}, for Poincar\'e polynomials of more general moduli spaces of stable quiver representations.

\section{The moduli space}\label{sec:moduli-space}
\subsection{The functor of points}
Fix an algebraically closed field $\bfk$, and integers $r,d>0$ and $n\geq 0$. Write $A=\bfk \braket{x_1,\ldots,x_d}$ for the free algebra on $d$ generators. Given a $\bfk$-scheme $B$, let $A_B$ denote the sheaf of $\OO_B$-algebras associated to the presheaf $A\otimes_{\bfk}\OO_B = \OO_B\braket{x_1,\ldots,x_d}$. Consider the functor $\mathsf F_{r,d}^n\colon \Sch_{\bfk}^{\opp} \to \Sets$ sending a $\bfk$-scheme $B$ to the set of equivalence classes of triples $(M,p,\beta)$ where
\begin{enumerate}
    \item $M$ is a left $A_B$-module, locally free of finite rank $n$ over $\OO_B$,
    \item $p\colon A_B^{\oplus r} \onto M$ is an $A_B$-linear epimorphism, and
    \item $\beta \subset \Gamma(B,M)$ is a basis of $M$ as an $\OO_B$-module.
\end{enumerate}
Two triples $(M,p,\beta)$ and $(M',p',\beta')$ are declared to be equivalent whenever there is a commutative diagram
\begin{equation}\label{diag:iso-class}
\begin{tikzcd}
A_B^{\oplus r} \arrow[equal]{d}\arrow[two heads]{r}{p} &
M\arrow{d}{\varphi} \\
A_B^{\oplus r}\arrow[two heads]{r}{p'} &
M'
\end{tikzcd}
\end{equation}
where $\varphi$ is an $\OO_B$-linear isomorphism sending $\beta$ onto $\beta'$.

There is another functor $\mathsf G_{r,d}^n\colon \Sch_{\bfk}^{\opp} \to \Sets$ sending $B$ to the set of equivalence classes of \emph{pairs} $(M,p)$ just as above, but where no basis of $\Gamma(B,M)$ is chosen, and no further condition is imposed besides the commutativity of Diagram \eqref{diag:iso-class}. Note that, by considering the kernel of a surjection, this functor parametrises \emph{left $A_B$-submodules} $K \subset A_B^{\oplus r}$ with locally free cokernel (and with just equality as equivalence relation).
As we shall see in the next subsection, both functors are representable.

\subsection{Definition of the noncommutative Quot scheme}
Fix a $\bfk$-vector space of dimension $n$, denoted $V_n$. Form the $(dn^2+rn)$-dimensional affine space
\begin{equation}\label{eqn:affine-space}
\BA = \End_{\bfk}(V_n)^d \times V_n^{r}.
\end{equation}
Define the \emph{span} of a point $x = (T_1,\ldots,T_d,v_1,\ldots,v_r) \in \BA$ to be the vector subspace
\[
\Span_{\bfk}(x) = \Span_{\bfk}\Set{T_1^{a_1}\cdots T_d^{a_d}\cdot v_j | a_i \geq 0, 1\leq j\leq r} \subset V_n.
\]
Let $U^n_{r,d} \into \BA$ be the open subset of points $x$ such that $\Span_{\bfk}(x) = V_n$. There is a $\GL_n$-action 
\[
g\cdot \bigl(T_1,\ldots,T_d,v_1,\ldots,v_r\bigr) = \left(gT_1g^{-1},\ldots,gT_dg^{-1},gv_1,\ldots,gv_r\right)
\]
on the whole of $\BA$, and $U^n_{r,d}$ agrees with the set of GIT stable points in $\BA$ with respect to the character $\det \colon \GL_n \to \bfk^\times$.

\begin{lemma}
The quotient map
\[
\begin{tikzpicture}
\node (a) at (0, 1.5)  {$U^n_{r,d}$};
\node (b) at (1.2, 0)    {$U^n_{r,d}/\GL_n = \BA \sslash_{\det} \GL_n$};
\node (c) at (0.2,0.8) {$\pi$};
\draw[->] (0, 1.2) -- (0, 0.3);
\end{tikzpicture}
\]
is a principal $\GL_n$-bundle. The GIT quotient $U^n_{r,d}/\GL_n$
is a smooth quasiprojective $\bfk$-variety of dimension $(d-1)n^2+rn$.
\end{lemma}

\begin{proof}
By general GIT, $\pi\colon U^n_{r,d}\to U^n_{r,d}/\GL_n$ is a geometric quotient.
In particular, every point $x \in U^n_{r,d}$ has closed $\GL_n$-orbit. We next observe that $x$ has trivial stabiliser. Indeed, if $g\in\GL_n$ fixes a point  $(T_1,\ldots,T_d,v_1,\ldots,v_r)\in U^n_{r,d}$, then $v_1,\ldots,v_r \in \ker(g-\id)\subset V_n$, but the smallest $\GL_n$-invariant subspace of $V_n$ containing $v_1,\ldots,v_r$ is $V_n$ itself. Thus $g = \id$. Thus, by an application of Luna's \'etale slice theorem (cf.~\cite[Cor.~4.2.13]{modulisheaves}), $\pi$ is a principal $\GL_n$-bundle in a neighbourhood of $\pi(x)$. Smoothness of $U^n_{r,d}/\GL_n$ now follows from smoothness of $U^n_{r,d}$ combined with smoothness and surjectivity of $\pi$. The dimension is easily computed as $\dim U^n_{r,d}-\dim \GL_n$.
\end{proof}

\begin{definition}
The geometric quotient $U^n_{r,d}/\GL_n$ is called the \emph{noncommutative Quot scheme}, and is denoted $\ncQuot^n_{r,d}$. When $r=1$, this is the more classical \emph{noncommutative Hilbert scheme}, and we write $\ncHilb^n_d$ instead of $\ncQuot^n_{1,d}$.
\end{definition}

\begin{prop}
The smooth quasiaffine scheme $U^n_{r,d}$ represents the functor $\mathsf F^n_{r,d}$, whereas the noncommutative Quot scheme $\ncQuot^n_{r,d}$ represents the functor $\mathsf G^n_{r,d}$.
\end{prop}

\begin{proof}
Full details in the case $d=3$ are given in \cite[Theorem 2.5]{BR18}. The proof generalises in a straightforward way to arbitrary $d$.
\end{proof}

\subsection{Quiver description}
The affine space \eqref{eqn:affine-space} can be seen as the moduli space of $(1,n)$-dimensional representations of the quiver $Q_{r,d}$, called the $r$-\emph{framed} $d$-\emph{loop quiver}, depicted in \Cref{fig:quiver11}.

\begin{figure}[H]
\centering
\begin{tikzpicture}[>=stealth,shorten >=2pt,looseness=.5,auto]
  \matrix [matrix of math nodes,
           column sep={3cm,between origins},
           row sep={3cm,between origins},
           nodes={circle, draw, minimum size=5.5mm}]
{ 
|(A)| \infty & |(B)| 0 \\         
};
\tikzstyle{every node}=[font=\small\itshape]

\node [anchor=west,right] at (-1.69,1.8) {$\bfk$};
\node [anchor=west,right] at (1.25,1.8) {$V_n$};
\node [anchor=west,right] at (-0.1,0.13) {$\vdots$};
\node [anchor=west,right] at (-0.2,0.7) {$v_1$};              
\node [anchor=west,right] at (-0.2,-0.7) {$v_r$}; 
\node [anchor=west,right] at (2,0.7) {$T_1$}; 
\node [anchor=west,right] at (2,-0.65) {$T_d$}; 
\node [anchor=west,right] at (2.1,0.13) {$\vdots$}; 

\draw (-1.45,1.6) -- (-1.45,0.45) {};
\draw (1.55,1.6) -- (1.55,0.4) {};
\draw[->] (A) to [bend left=25,looseness=1] (B) node [midway] {};
\draw[->] (A) to [bend right=25,looseness=1] (B) node [midway,below] {};
\draw[->] (B) to [out=30,in=60,looseness=8] (B) node {};
\draw[->] (B) to [out=330,in=300,looseness=8] (B) node {};
\end{tikzpicture}
\caption{The $r$-framed $d$-loop quiver $Q_{r,d}$. The vertex `$0$' (resp.~the vertex `$\infty$') carries dimension $n$ (resp.~dimension $1$).}\label{fig:quiver11}
\end{figure}
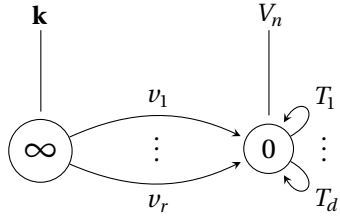

It turns out (see \cite[Prop.~2.4]{BR18} for a proof in the case $d=3$, but extending to arbitrary $d$) that $U_{r,d}^n$ agrees with the open subscheme of $\zeta_n$-stable representations, where $\zeta_n=(-n,1)$ and stability is King stability \cite{King}.

\section{Computation}\label{sec:computation}
In this section we prove \Cref{eqn:main-formula}.

The Grothendieck ring of $\bfk$-varieties $K_0(\Var_{\bfk})$ is the free abelian group on isomorphism classes $[Y]$ of $\bfk$-varieties, modulo the relations $[Y] = [Y\setminus Z]+[Z]$, for all closed subvarieties $Z\into Y$. The ring structure is determined by $[Y]\cdot [Y'] = [Y\times_{\bfk}Y']$. The main rules for `playing' with $K_0(\Var_{\bfk})$ are the following:
\begin{enumerate}
    \item [\mylabel{rule_i}{(i)}] If $Y \to B$ is a bijective morphism of $\bfk$-varieties, then $[Y] = [B]$, and
    \item [\mylabel{rule_ii}{(ii)}] If $Y \to B$ is a Zariski locally trivial fibration with fibre $F$, then $[Y] = [B] \cdot [F]$,
\end{enumerate}
There is also a Grothendieck ring of stacks (with affine geometric stabilisers) $K_0(\St_{\bfk})$. The definition is due to Ekedahl \cite{EKStacks}, as well as the proof (Theorem 1.2 in loc.~cit.) that the natural homomorphism $K_0(\Var_{\bfk})\to K_0(\St_{\bfk})$ agrees with the localisation at the multiplicative system $S\subset K_0(\Var_{\bfk})$ generated by $\BL$ and $\BL^m-1$ for $m>0$. This, in turn, is equivalent to localising at the classes of special algebraic groups, resp.~at the classes $[\GL_m]$ for all $m>0$, see \cite[Lemma 3.8]{Bri-Hall}. We take this as our working definition of $K_0(\St_{\bfk})$.

Thus, if $G$ is a special algebraic group, such as $\GL_m$, then $[G]$ becomes invertible in $K_0(\St_{\bfk})$. We then have the third crucial rule of the game:
\begin{enumerate}
    \item [\mylabel{rule_iii}{(iii)}] If $Y \to B$ is a principal $G$-bundle in the category of varieties (or schemes, or algebraic spaces), then $[B]=[Y]/[G]\in K_0(\St_{\bfk})$.
\end{enumerate}

The motivic class of the Grassmannian of $k$-planes in $V_n\cong \bfk^n$ is
\begin{equation}\label{motive_grass}
[G(k,V_n)] = [G(k,n)] = \frac{[n]_{\BL}}{[k]_{\BL}[n-k]_{\BL}},
\end{equation}
where, as in the introduction, we have set $[n]_{\BL} = \prod_{1\leq i\leq n}(\BL^i-1)$ and $\BL=[\BA^1_{\bfk}]$ is the \emph{Lefschetz motive}.

\smallbreak
Back to our calculation.  We keep the previous notation.
We have a bijective morphism
\[
\coprod_{k=0}^n Y_{r,d}^k \to \BA,
\]
where $Y_{r,d}^k \into \BA = \End_{\bfk}(V_n)^d \times V_n^{r}$ is the locally closed subscheme of points $x$ such that $\Span_{\bfk}(x) \subset V_n$ is $k$-dimensional. In particular, we have $Y_{r,d}^n = U_{r,d}^n$. 
By Rule \ref{rule_i}, this locally closed stratification induces an identity
\[
\BL^{dn^2+rn} = \sum_{k=0}^{n}\, [Y_{r,d}^k] = [U^n_{r,d}] + \sum_{k=0}^{n-1}\, [Y_{r,d}^k].
\]
The morphism to the Grassmannian
\[
\mathsf{span}_k\colon Y_{r,d}^k \to G(k,V_n)
\]
sending $x \mapsto \Span_{\bfk}(x)$ is Zariski locally trivial,\footnote{This observation, as far as we know, first appeared in the proof of Theorem 3.7 in \cite{BBS}.} and we denote by $F_{r,d}^k$ its fibre. Let $\Lambda \in G(k,V_n)$ be a point. Fix a basis of $V_n$ whose first $k$ vectors belong to $\Lambda \subset V_n$. Then, every point in the fibre $F_{r,d}^k = \mathsf{span}_k^{-1}(\Lambda)$ can be represented as a tuple $(T_1,\ldots,T_d,v_1,\ldots,v_r)$ of the form
\[
T_i = 
\begin{pmatrix}
D_i & A_i \\
0 & D'_i
\end{pmatrix}, \quad 
v_j = 
\begin{pmatrix}
w_j \\
0
\end{pmatrix}
\]
for $i=1,\ldots,d$ and $j=1,\ldots,r$, where
\begin{itemize}
    \item [$\circ$] $D_i$ is a $k\times k$ matrix,
    \item [$\circ$] $A_i$ is a $k \times (n-k)$ matrix, 
    \item [$\circ$] $D'_i$ is an $(n-k)\times (n-k)$ matrix,
    \item [$\circ$] $w_j \in V_k$ is a $k$-vector,
    \item [$\circ$] $(D_1,\ldots,D_d,w_1,\ldots,w_r)\in U_{r,d}^k \subset \End_{\bfk}(V_k)^d \times V_k^r$.
\end{itemize}
The morphism
\[
F_{r,d}^k \to U_{r,d}^k
\]
forgetting $A_1,\ldots,A_d$ and $D'_1,\ldots,D'_d$ is locally trivial with fibre
\[
\Mat_{k\times (n-k)}(\bfk)^d \times \Mat_{(n-k) \times (n-k) }(\bfk)^d \cong \BA^{dk(n-k) + d(n-k)^2} = \BA^{dn(n-k)}.
\]
Using Rule \ref{rule_ii} twice, along with \Cref{motive_grass}, we obtain
\[
[Y^k_{r,d}] = [G(k,V_n)]\cdot [F_{r,d}^k] = \frac{[n]_{\BL}}{[k]_{\BL}[n-k]_{\BL}}\cdot [U_{r,d}^k]\cdot \BL^{dn(n-k)}.
\]
We then compute (using Rule \ref{rule_iii} for the first identity)
\begin{align*}
    [\ncQuot_{r,d}^n]
\,\,&=\,\, \frac{[U_{r,d}^n]}{[\GL_n]} \\
\,\,&=\,\, \frac{1}{[\GL_n]}\left(\BL^{dn^2+rn}-\sum_{k=0}^{n-1}\,[Y^k_{r,d}]\right) \\
\,\,&=\,\, \frac{1}{[\GL_n]}\left(\BL^{dn^2+rn}-\sum_{k=0}^{n-1}\, \frac{[n]_{\BL}}{[k]_{\BL}[n-k]_{\BL}}\cdot [U_{r,d}^k]\cdot \BL^{dn(n-k)} \right) \\
\,\,&=\,\,\BL^{dn^2}\left(\frac{\BL^{rn}}{[\GL_n]}-\sum_{k=0}^{n-1}\, \frac{[U_{r,d}^k]}{[k]_{\BL}[n-k]_{\BL}\BL^{\binom{n}{2}+dkn}}\right) \\
\,\,&=\,\,\BL^{dn^2}\left(\frac{\BL^{rn}}{[\GL_n]}-\sum_{k=0}^{n-1}\, \frac{[U_{r,d}^k]}{[k]_{\BL}\displaystyle\BL^{\binom{k}{2}}[n-k]_{\BL}\displaystyle\BL^{\binom{n-k}{2}}\displaystyle\BL^{kn(d+1)-k^2}} \right) \\
\,\,&=\,\,\BL^{dn^2}\left(\frac{\BL^{rn}}{[\GL_n]}-\sum_{k=0}^{n-1}\, \frac{[\ncQuot_{r,d}^k]}{[\GL_{n-k}]}\BL^{k^2-kn(d+1)}\right).
\end{align*}
Using the relation $\binom{n}{2}-\binom{n-k}{2} = \binom{k}{2} + k(n-k)$ one easily obtains
\begin{align*}
\frac{1}{[\GL_{n-k}]} 
&= \frac{\BL^{\binom{k}{2}+k(n-k)}}{[\GL_n]}  \frac{[n]_{\BL}}{[n-k]_{\BL}} \\
&= \frac{\BL^{\binom{k}{2}+k(n-k)}}{[\GL_n]}[k]_{\BL}[G(k,n)] \\
&= \frac{\BL^{k(n-k)}}{[\GL_n]}[\GL_k][G(k,n)],
\end{align*}
from which one can write
\[
[\ncQuot^n_{r,d}] = \frac{\BL^{dn^2}}{[\GL_n]}\left(\BL^{rn} - \sum_{k=0}^{n-1}\, [\ncQuot^k_{r,d}][\GL_k][G(k,n)]\BL^{-dnk} \right).
\]
Thus \Cref{eqn:main-formula} in \Cref{thm:main-thm} is proved. The statement on functional equations will be proved in \Cref{prop:functional-eqn}.

\begin{remark}
Rearranging terms, the identity we just proved is equivalent to the relation
\[
\BL^{rn} = \sum_{k=0}^n\,[\ncQuot^k_{r,d}][\GL_k][G(k,n)]\BL^{-dnk},
\]
but unfortunately this is not giving rise to an explicit formula for the generating function of the motives $[\ncQuot^n_{r,d}]$. However, note that the right hand side turns out to be independent on $d$.
\end{remark}

We close the section with a few explicit examples.

\begin{example}
If $n=1$, we have $\ncQuot^1_{r,d} = \Quot_{\BA^d}(\OO^{\oplus r},1) =\BA^d \times \BP^{r-1} $, and so 
\[
[\ncQuot^1_{r,d}] = \BL^d\bigl(1+\BL+\cdots+\BL^{r-1}\bigr).
\]
\end{example}

\begin{example}
If $d=1$, there is no difference between commutative and noncommutative, i.e.
\[
\Quot_{\BA^1}(\OO^{\oplus r},n) = \ncQuot^n_{r,1},
\]
because there is only one matrix involved and so the commutativity condition $[T,T]=0$ is redundant.
See \cite{BFP19,ricolfi2019motive} for a formula expressing the motive of the Quot scheme of points on a smooth curve. Thus the generating series is
\begin{equation}\label{eqn:d=1}
\sum_{n\geq 0}\,[\ncQuot^n_{r,1} ] t^n = \zeta_{\BP^{r-1}}(\BL t) = \prod_{i=0}^{r-1} \frac{1}{1-\displaystyle\BL^{i+1}t},
\end{equation}
where $\zeta_Y(t) = \sum_{n\geq 0} [\Sym^n Y]t^n$ is the Kapranov motivic zeta function of a variety $Y$.
\end{example}

\begin{example}
If $d=r=1$, then \Cref{eqn:d=1} specialises to
\begin{equation}\label{eqn:d=r=1}
\sum_{n\geq 0}\,[\ncQuot^n_{1,1} ] t^n = \sum_{n\geq 0}\,\BL^nt^n = \frac{1}{1-\BL t}.
\end{equation}
Therefore our formula is equivalent to the formal identities
\begin{align*}
x^n 
&= x^{n^2}\left(\frac{x^n}{\displaystyle x^{\binom{n}{2}}[n]_x}-\sum_{k=0}^{n-1}\frac{x^k}{\displaystyle x^{\binom{n-k}{2}}[n-k]_x}x^{k^2-2kn} \right) \\
&=\frac{x^{n^2}}{x^{\binom{n}{2}}[n]_x} \left(x^n - \sum_{k=0}^{n-1} \frac{[n]_{x}}{[n-k]_{x}} x^{k-nk+\binom{k}{2}}\right),   
\end{align*}
which can be verified directly for all $n\geq 0$.
\end{example}

\section{Generating functions}
In this section we fix $\bfk = \BC$, since we rely on \cite{Reineke_2005}. For a smooth $N$-dimensional variety $M$, set $[M]_{\vir} = \BL^{-N/2}[M]$, as in \cite{BBS}. Set $N_{n,r,d} = \dim \ncQuot^{n}_{r,d} = (d-1)n^2+rn$ and consider the generating function
\[
\mathsf Z_{r,d}(t) = \sum_{n=0}^\infty\,  [\ncQuot^{n}_{r,d}]_{\vir} (\BL^{(r-1)/2}t)^n.
\]
The following proposition completes the proof of \Cref{thm:main-thm}.

\begin{prop}\label{prop:functional-eqn}
The rank $1$ series $\mathsf Z_{1,d}(t)$ satisfies the functional equation
\begin{equation}\label{eqn:Z1}
\mathsf Z_{1,d}(t) = 1+\BL^{d/2}t\cdot \prod_{i=0}^{d-1}\mathsf Z_{1,d}(\BL^it).
\end{equation}
The rank $r$ series $\mathsf Z_{r,d}(t)$ factors as
\begin{equation}\label{eqn:Zr}
\mathsf Z_{r,d}(t) = \prod_{i=0}^{r-1} \mathsf Z_{1,d}(\BL^{i}t).
\end{equation}
\end{prop}

\begin{proof}
This is a formal consequence of \cite[Thm.~1.4]{Reineke_2005}, using the fact that the smooth varieties $\ncQuot^n_{r.d}$ admit a cell decomposition (and hence they have vanishing odd cohomology). For such varieties, the motive in $K_0(\Var_{\BC})$ is the Poincar\'e polynomial in Borel--Moore homology evaluated at the Lefschetz motive.
For a smooth $N$-dimensional variety $M$ with a cellular decomposition, the Poincar\'e polynomials
\begin{align*}
P(M,q) &= \sum_{k\geq 0} h^{k}(M,\BQ)q^{-k/2} \\
P^{\mathrm{BM}}(M,q) &= \sum_{k\geq 0}h^{\mathrm{BM}}_{k}(M,\BQ)q^{k/2}
\end{align*} 
satisfy, thanks to the Poincar\'e duality isomorphisms $\HH_k^{\mathrm{BM}}(M,\BQ) \cong \HH^{2N-k}(M,\BQ)$, the relation
\[
P(M,q) = q^{-N}P^{\mathrm{BM}}(M,q),
\]
and hence the motivic identity
\[
P(M,\BL) = \BL^{-\dim M}[M] \in K_0(\Var_{\BC})[\BL^{-1}].
\]
We now illustrate how to derive Equation \eqref{eqn:Z1} out of \cite[Thm.~1.4]{Reineke_2005}. 
Reineke considers in loc.~cit.~the generating series
\[
\overline\zeta^{(d)}_r(q,t) = \sum_{n=0}^\infty q^{(d-1)\binom{n}{2}+(r-1)n}P(\ncQuot^n_{r,d},q)t^n,
\]
and proves the identities
\begin{align*}
    \overline\zeta^{(d)}_1(q,t) &= 1+t\prod_{i=0}^{d-1}\overline\zeta^{(d)}_1(q,q^it) \\
    \overline\zeta^{(d)}_r(q,t) &= \prod_{i=0}^{r-1} \overline\zeta^{(d)}_1(q,q^it).
\end{align*}
When $r = 1$, one evaluates
\begin{align*}
\overline\zeta^{(d)}_1(\BL,t) 
&= \sum_{n=0}^\infty \,\BL^{(d-1)\binom{n}{2}-N_{n,1,d}}[\ncHilb^n_d]t^n \\
&= \sum_{n=0}^\infty\, \BL^{-N_{n,1,d}/2}[\ncHilb^n_d](\BL^{-d/2}t)^n \\
&=\sum_{n=0}^\infty\, [\ncHilb^n_d]_{\vir}(\BL^{-d/2}t)^n.
\end{align*}
Therefore
\begin{align*}
\mathsf Z_{1,d}(t) 
&= \overline\zeta^{(d)}_1(\BL,\BL^{d/2}t) \\
&=1+\BL^{d/2}t\prod_{i=0}^{d-1}\overline\zeta^{(d)}_1(\BL,\BL^{i+d/2}t) \\
&=1+\BL^{d/2}t\prod_{i=0}^{d-1}\mathsf Z_{1,d}(\BL^i t).
\end{align*}
Similarly, to confirm \Cref{eqn:Zr}, we compute
\begin{align*}
    \mathsf Z_{r,d}(t)
&= \sum_{n=0}^\infty \, \BL^{-(N_{n,r,d}+(1-r)n)/2} [\ncQuot^n_{r,d}] t^n \\
&= \sum_{n=0}^\infty \, \BL^{-N_{n,1,d}/2} [\ncQuot^n_{r,d}] t^n \\
&= \sum_{n=0}^\infty \, \BL^{(d-1)\binom{n}{2}+(r-1)n-N_{n,r,d}}[\ncQuot^n_{r,d}](\BL^{d/2}t)^n\\
&= \overline\zeta^{(d)}_r(\BL,\BL^{d/2}t)\\
&= \prod_{i=0}^{r-1}\overline\zeta^{(d)}_1(\BL,\BL^{i+d/2}t)\\
&= \prod_{i=0}^{r-1}\mathsf{Z}_{1,d}(\BL^it).\qedhere
\end{align*}
\end{proof}

The proof of \Cref{thm:main-thm} is now complete.

\begin{remark}
    The functional equation for higher rank noncommutative Quot schemes fits nicely in the philosophical framework according to which rank $r$ theories factor as products of $r$ rank $1$ theories, suitably shifted \cite{cazzaniga2020higher,Cazzaniga:2020aa,Monavari:2021tl}.
\end{remark}

\begin{remark}
Unravelling the functional equations of \Cref{prop:functional-eqn}, one finds yet \emph{another} recursion fulfilled by the motives $[\ncQuot^n_{r,d}]$ computed in \Cref{eqn:main-formula}. It would be interesting to find a geometric interpretation of such recursion.
\end{remark}

\subsection*{Acknowledgments}
We thank Markus Reineke for very helpful discussions around his previous work \cite{Reineke_2005,Engel_2008}.

\bibliographystyle{amsplain}
\bibliography{bib}

\bigskip
\bigskip
\bigskip
\noindent
Andrea T. Ricolfi \\
\address{SISSA, Via Bonomea 265, 34136, Trieste (Italy)} \\
\texttt{aricolfi@sissa.it}

\end{document}